\documentclass[a4paper, 11pt,russian]{article}
\usepackage[utf8]{inputenc}
\usepackage[english]{babel}
\usepackage{amssymb}
\usepackage{ifthen}
\usepackage{amsmath}

\makeatletter
\def\@seccntformat#1{\csname the#1\endcsname.\ } 
\def\@biblabel#1{#1.} 
\makeatother

\date{}

\newenvironment{proof}[1][\hspace{-1.0ex}]%
{\par\addvspace{1mm}{\it Proof\hspace{1.0ex}{#1}.} }%
{\quad$\blacktriangle$\par\addvspace{1mm}}

    \newif\ifNoRemark
    \def\addtheorem#1#2#3#4{ 
    \ifthenelse{\expandafter\isundefined\csname the#2\endcsname}{\newcounter{#2}}{}
    \newenvironment{#1}[1][\global\NoRemarktrue]
     {\par\addvspace{2mm}\noindent 
       \refstepcounter{#2}{\bf #3~\csname the#2\endcsname
      \vphantom{##1}\ifNoRemark.\ \else\ (##1).\fi}\begingroup #4}%
     {\endgroup\par\addvspace{1mm}\global\NoRemarkfalse}
    \expandafter\newcommand\csname b#1\endcsname{\begin{#1}}
    \expandafter\newcommand\csname e#1\endcsname{\end{#1}}
    }

\addtheorem{proposition}{proposition}{Proposition}{\sl}
\addtheorem{theorem}{theorem}{Theorem}{\sl}
\addtheorem{lemma}{lemma}{Lemma}{\sl}
\addtheorem{corol}{corol}{Corollary}{\sl}

\title{On the number of SQSs, latin hypercubes and MDS codes\thanks{ The work was funded by the
Russian Science Foundation (grant No 14-11-00555). }}

\author{Vladimir N. Potapov\\
\emph{Sobolev Institute of Mathematics, Novosibirsk, Russia}
\\vpotapov@math.nsc.ru}

\begin{document}

\maketitle

\begin{abstract}

It is established that the logarithm of the  number of latin
$d$-cubes of order $n$ is $\Theta(n^{d}\ln n)$ and the logarithm of
the number of pairs of orthogonal latin squares of order $n$ is
$\Theta(n^2\ln n)$. Similar estimations are obtained for
 systems of mutually strong orthogonal latin $d$-cubes.
As a consequence, it is  constructed a set of Steiner quadruple
systems of order $n$ such that  the logarithm of its cardinality is
$\Theta(n^3\ln n)$ as $n\rightarrow\infty$ and $n\ {\rm mod}\ 6= 2\
{\rm or}\ 4$.

\end{abstract}


MSC 51E10, 05B05, 05B15

\section{Introduction}

 Let  $Q$ denote the set $\{0,1,\dots,q-1\}$.
A {\it Steiner system} with parameters $\tau, d, q$, $\tau\leq d$,
written $S(\tau,d,q)$, is  a set of $d$-element subsets of $Q$
(called {\it blocks}) with the property that each $\tau$-element
subset of $Q$ is contained in exactly one block. The best known
Steiner systems are $S(2,3,q)$ (called {\it Steiner Triple Systems
of order} $q$, $STS(q)$) and $S(3,4,q)$ (called {\it Steiner
Quadruple Systems of order} $q$, $SQS(q)$). It is well known that an
$STS(q)$ exists if and only if $q\ {\rm mod}\ 6= 1\ {\rm or}\ 3$.
Alekseev \cite{Alek} showed that the logarithm of the number of
nonisomorphic $STS(q)$ is equal\footnote{\, Notation
$f(x)=\Theta(g(x))$ as $x\rightarrow x_0$ means that there exist
constants $c_2\geq c_1>0$ and a  neighborhood $U$ of $x_0$ such that
for all $x\in U$ $c_1g(x)\leq f(x)\leq c_2g(x)$.} to $\Theta(q^2\ln
q)$ as $q\rightarrow\infty$. Egorychev \cite{Egor} obtain the
asymptotic of this function, which  is $\frac{q^2}{6}\ln q$. More
accurate an upper bound of the number of STSs  was calculated in
\cite{LL13}. Hanani \cite{Hanani} proved that the necessary
condition  of $q\ {\rm mod}\ 6= 2\ {\rm or}\ 4$ for the existence of
 Steiner quadruple systems of order $q$ is also sufficient. Lenz
\cite{Lenz} proved that the logarithm of the number of different
SQSs of order $q$ is greater than $cn^3$ where $c>0$ is a constant
and $n$ is admissible. Constructions and properties of SQSs are
studied in \cite{Hartman,LR,ZG,ZZ}. Recently Keevash \cite{Keevash}
showed  that the natural divisibility conditions are sufficient for
existence of Steiner system apart from a finite number of
exceptional $q$ given fixed $\tau$ and $d$. Moreover, Keevash
\cite{Keevash} proved that the logarithm of the number of different
$S(\tau,d,q)$ (with fit parameters) is equal to ${d \choose
\tau}^{-1}{q \choose \tau}(d-\tau)(1+o(1))\log q$ as
$q\rightarrow\infty$. His proof is based on probabilistic methods.
 One of the results in this
paper is the construction of a set of SQSs that reaches asymptotic
estimation  $\Theta(q^3\ln q)$ as $q\rightarrow\infty$ and $q\ {\rm
mod}\ 6= 2\ {\rm or}\ 4$.

A {\it latin square} of order $q$ is a $q\times q$ array of $q$
symbols  in which each symbol occurs exactly once in each row and in
each column.
A $d_0$-dimensional array of $q$ symbols is called a {\it latin
$d_0$-cube of order $q$} if each $2$-dimensional subarray is a latin
square of order $q$.
 The best known asymptotic estimate
of the number of latin squares is $((1+o(1))q/e^2)^{q^{2}}$ (see
\cite{LW}), which follows from the lower bound obtained in
\cite{Egor} and the upper bound that can be derived from Bregman's
inequality for permanents. A table of numbers of nonequivalent latin
squares of small orders is available in \cite{HKO,MW0,Ston}. An
upper bound $((1+o(1))q/e^{d_0})^{q^{d_0}}$ of the number of latin
$d_0$-cubes is proved in \cite{LL14}.  One of the goal in this paper
is to calculate the logarithm of the number  of latin $d_0$-cubes of
order $q$. It is equal to $\Theta(q^{d_0}\ln q)$ as
$q\rightarrow\infty$. In  \cite{KPS} and \cite{KP11} it is found
estimations of the number of latin $d_0$-cubes of order $q$ as
$d\rightarrow\infty$ and $q$ is a constant. A classification of
latin hypercubes for small orders and dimensions is available in
\cite{MW}.

From the definition we can be sure that a latin $d_0$-cube is the
$d_0$-ary Cayley table of a quasigroup.  A system consisting of $t$
$s$-ary functions $f_1,\dots,f_t$  ($t\geq s$) is {\it orthogonal},
if for each subsystem $f_{i_1},\dots,f_{i_s}$ consisting of  $s$
functions it holds $
\{(f_{i_1}(\overline{x}),\dots,f_{i_s}(\overline{x})) \ |\
\overline{x}\in Q^{s}\}=Q^{s}.$ If the system remains orthogonal
after substituting any constants for each subset of variables then
it is called {\it strongly orthogonal} (see \cite{EMullen}). If the
number of variables  equals $2$ ($s=2$) then such a system is
equivalent to a set of {\it Mutually Orthogonal Latin Squares}
(MOLS). If $s>2$, it is a set of {\it Mutually Strong Orthogonal
Latin $s$-Cubes} (MSOLC). For example, if we have three orthogonal
latin cubes ($t=s=3$) then each triple $(a_1,a_2,a_3)\in Q^3$ meets
once as
$(a_1,a_2,a_3)=(f_1(x_1,x_2,x_3),f_2(x_1,x_2,x_3)f_3(x_1,x_2,x_3))$.
 Furthermore, if we fix one variable ($x_3=c$) then we obtain three
 MOLS $f_1(x_1,x_2,c), f_2(x_1,x_2,c)$ and $f_3(x_1,x_2,c)$.

 A subset $C$ of $Q^d$ is called an
$MDS(t,d,q)$ {\it code} (of order $q$, code distance $t+1$ and
length $d$) if $|C\cap\Gamma|=1$ for each $t$-dimensional
axis-aligned plane $\Gamma$. A system of $t$ MSOLC is equivalent to
an MDS code with distance $t+1$ (see \cite{EMullen}). For example,
$MDS(1,d,q)$ {\it codes} are equivalent to latin $(d-1)$-cubes of
order $q$. We prove that the logarithm of the number of sets of $t$
($t\geq 2$) MOLS of order $q$   is $\Theta(q^2\ln q)$ as
$q\rightarrow \infty$. Besides, we establish new  lower bounds for
the numbers of other systems of MSOLC with certain parameters. For a
subsequence of orders,  the logarithm of the number of systems of
MOLS was found in \cite{Donovan} (in other terminology but using a
similar method). A classification of system of MSOLC (and MDS codes)
for small orders and dimensions is available in \cite{Egan,KKO,KO}.

The main idea of the proof of the lower bound for the number of MDS
codes is as follows. Let $q=p^k$ where $p$ is a prime number. Then
we can consider elements of $Q$ as $k$-dimensional vectors over
$GF(p)$. Consider a linear (over $GF(p)$) MDS code or latin
hypercube of order $p^k$. This code contains a lot of subcodes or
switching components. In this case switching components are affine
subspaces over $GF(p)$. Then we calculate a number of possibilities
to obtain a new code from the initial code by switching disjoint
components. The method of switching components is discussed in
\cite{Ost}. We use results of \cite{KPS} to obtain a lower bound for
the number of latin hypercubes of large enough orders and results of
\cite{Dukes} and \cite{Zhu} to obtain a lower bound for the number
of pair of MOLS.

It is well known that STSs of order $q$ are equivalent to totally
symmetric quasigroups of order $q$. As mentioned above, a latin
square can be represented as the Cayley table of a quasigroup,
i.\,e. as a set of ordered triples of the $q$-element set such that
each pair of elements occurs in each pair of positions and a pair of
elements of any triple defines the third element of the triple. The
first and the second elements of triples define row and column, and
the third element defines the symbol in the corresponding entry of a
latin square. Thus,  given an $STS(q)$,  we can obtain a totally
symmetric latin square by replacing each unordered triple with the
six ordered triples and by adding $q$ triples of the form $(a,a,a)$.
Analogously, given an $SQS(q)$, we can obtain a symmetric latin cube
by replacing each unordered quadruple with the twenty four ordered
quadruple and by adding $3q(q-1)$ quadruples of the forms
$(a,a,b,b)$, $(a,b,a,b)$, $(a,b,b,a)$,  and $q$ quadruples of the
form $(a,a,a,a)$. The main idea of the proof of the lower bound for
the number of SQSs of order $q$ is the insertion an arbitrary latin
cube (of order less than $q$)  into an $SQS(q)$. We use  some
constructions of SQSs by Hanani \cite{Hanani63} and Hartman
\cite{Hartman}. Moreover, in one case ($q\mod 36=2$) we need to
introduce a new construction of SQSs which allows insert of
arbitrary latin cubes.

\section{MDS codes}

 The following criteria for MDS codes are well-known.

\begin{proposition}\label{proNM5}
A subset $M\subset Q^d$ is an {\rm MDS} code if and only if
$|M|=|Q|^{d-\varrho+1}$, where $\varrho$ is the code distance of
$M$.
\end{proposition}

\begin{proposition}\label{proAO0}{\rm \cite{EMullen}} A set $C\subset Q^{t+s}$
is an {\rm MDS}-code with code distance $\varrho_C=s+1$ if and only
if there exists strongly orthogonal system consisting of $t$ $s$-ary
quasigroups $f_1,\dots,f_t$ such that
$$C=\{(x_1,\dots,x_{s},f_1(\overline{x}),\dots,f_t(\overline{x})) \ |\ \overline{x}\in
Q^{s}\}.$$  \end{proposition}

A projection of a set $C\subset Q^d$ in the  $i$th direction is
called the set $$C_i=\{(x_1,\dots,x_{i-1},x_{i+1},\dots x_{d}) \ |\
\exists x_i\ \mbox{such that}\ (x_1,\dots,x_{d})\in C\}.$$

\begin{proposition}\label{progs1}
Any projection of an MDS code is an MDS code.
\end{proposition}

Proposition \ref{progs1} follows from Proposition \ref{proNM5}.

\begin{proposition}\label{progs2}
Let $M\subset Q^5$ be an MDS code with  code distance $4$ and let
$M'$ be a $4$-dimensional projection of $M$. Then there exists an
MDS code $C\subset Q^4$ with  code distance $2$ such that $M'\subset
C$.
\end{proposition}
\begin{proof} By results of \cite{EMullen} any MDS code correspond to
a system of orthogonal quasigroups. So $(x,y,u,v,w)\in M$ whenever
$\qquad \left\{
\begin{array}{l} u=f(x,y);\\
v=g(x,y);\\
w=h(x,y),\\
\end{array} \right.$

\noindent where $f,g,h$ determine a set of $3$ MOLS.

Determine $M'$ by equations $\qquad \left\{
\begin{array}{l} u=f(x,y);\\
v=g(x,y).\\
\end{array} \right.$

Define the function $\varphi:Q^2\rightarrow Q$ by
$\varphi(f(x,y),g(x,y))=h(x,y)$.
 The
orthogonality of $f$ and $g$ yields that the function $\varphi$ is
well defined; and the orthogonality  of $f$ and $h$, the
orthogonality of $g$ and $h$ provide that $\varphi$ is a quasigroup.
Hence the set $C=\{(x,y,u,v) \ | \varphi(u,v)=h(x,y)\}$ is an MDS
code and $M'\subset C$ by construction. \end{proof}

\begin{proposition}\label{progs11}{\rm \cite{Wilson}}
For every integer $t$ there is an integer  $k(t)$  such that for all
$k>k(t)$ there exists a set of $t$ MOLS of order $k$.

\end{proposition}

Note that  $k(6)$ is not greater  than $74$ (see \cite{HCD}, Table
3.81).

A subset $T$ of an MDS code $M\subset Q^d$ is called a {\it subcode}
or a {\it component} of the code if $T$ is an MDS code  in
$A_1\times\dots\times A_d$ with the same code distance as $M$ and
$T=M\cap (A_1\times\dots\times A_d)$ where $A_i\subset Q$, $i\in
\{1,\dots,d\}$. Obviously $|A_1|=\dots=|A_d|$ and $|A_1|$ is the
order of the subcode $T$.
 A definition of a latin subsquare is analogous.

Let us now consider  possible orders of subcodes.  The following
proposition is well-known for case of pairs of orthogonal latin
squares (a case of MDS code with distance $\varrho=3$).
\begin{proposition}
If  an {\rm MDS} code $M\subset Q^d$ with code distance $\varrho$,
$d>\varrho\geq 3$, contains a proper subcode of order $\ell$ then
$\varrho \leq \ell\leq |Q|/\varrho$.
\end{proposition}
\begin{proof}
By definition every strongly orthogonal system consisting of
$t=\varrho-1$ functions includes a  system $f_1,\dots,f_t$ of $t$
MOLS. A system of MOLS of order $\ell$ consists of not more than
$\ell-1$ latin squares. Therefore $t\leq \ell-1$. Without loss of
generality we can assume that the subcode includes a system of $t$
MOLS of order $\ell$ over the alphabet $B$. Denote by $b$ the
symbols of $B$ and by $a$ the other symbols. By the definition of
orthogonal system, for any pair $a$, $b$ and any distinct $i,j\in
\{1,\dots,t\}$, there exists $(u_1,u_2)\in (Q\setminus B)^2$ such
that $f_i(u_1,u_2)=a$, $f_j(u_1,u_2)=b$ and $f_{i'}(u_1,u_2)\in
Q\setminus B$ for $i'\neq i,j$. Thus $|(Q\setminus
B)^2|=(|Q|-\ell)^2\geq t\ell(|Q|-\ell)$.
 \end{proof}

 A Latin square $f$ is called symmetric  if
$f(x,y)=f(y,x)$ for each $x,y$. It is called unipotent if $f(x,x)=0$
for every $x$. By using the construction from \cite{Cameron} it is
easy to prove

\begin{proposition}\label{progs12}
Let $q$ be even and $\ell\leq q/4$. Then there is a symmetric
unipotent latin square of order $q$ with subsquare of order $\ell$
in $K_0\times K_1\times K_1$ and $K_1\times K_0\times K_1$, where
$K_0=\{0,\dots,\ell-1\}$ and $K_1=\{q-\ell,\dots,q-1\}$.
\end{proposition}

\section{Lower bound for the number of MDS codes}

Let $Q$ be a finite field. An MDS code $C$ is called {\it linear
(affine)} if it is a linear (or affine) subspace of $Q^d$. In this
case the functions $f_1,\dots,f_t$ (defined in Proposition
\ref{proAO0}) are linear and rank of the code is equal to ${\rm
dim}(C)= s$. Let $F$ be a subfield of a finite field $Q$ and
$|Q|=|F|^k$. Then we can consider $Q$ as $k$-dimensional vector
space over $F$. We will call $C\subset Q^d$ a linear code over $F$
if it is linear (i.~e.
$f_i=\alpha_{1i}x_{1}+\dots+\alpha_{di}x_{d}$) and all coefficients
$\alpha_{ji}$ ($j=1,\dots,d$, $i=1,\dots,t$) are in $F$. For $a,v\in
Q$ denote by $L(a,v)=\{a+\alpha v \ |\ \alpha\in F\}$ an
$1$-dimensional affine subspace in $Q$.
A number of different $L(a,v)$ for fixed $a\in Q$ is equal to the
number of 1-dimensional linear subspace of $F^k$, i.~e.
$\frac{|F|^k-1}{|F|-1}$.

By using a well-known construction of a linear MDS code (\cite{MS},
Chapters 10,11) with matrix  over prime subfield $GF(p)$ we can
conclude that the following proposition is true.

\begin{proposition}\label{proNMb}
Let $p$ be a prime number, let $d,k$ be integers, and  let
$|Q|=p^k$. Then

\noindent {\rm (a)} for each  $\varrho \in \{2,d\}$ there exists a
linear over $GF(p)$ {\rm MDS} code $C\subset Q^d$ with code distance
$\varrho$,

\noindent {\rm (b)} for each  integers $d\leq p+1$  and $\varrho$,
$3\leq \varrho\leq p$, there exists a linear over $GF(p)$ {\rm MDS}
code $C\subset Q^d$ with code distance $\varrho$.
\end{proposition}

If $2<\varrho<d$   then the length $d$ of  a linear MDS code of
order $p^k$ with code distance $\varrho$ does not exceed $p^k+1$ for
$p\neq 2$ or $p^k+2$ for $p=2$ (see \cite{Ball}, \cite{BallBeu}).


\begin{proposition}\label{proNMa}
Assume $C$ is a code with a subcode $C_1$ of order $\ell$ and  a
code $C_2$ has the same parameters as $C_1$. Then it is possible to
exchange $C_1$ by $C_2$ in $C$ and to obtain the code $C'$ with the
same parameters as $C$.
\end{proposition}

 It is said the codes $C$ and $C'$ obtained from each other by
switching \cite{Ost}. If a code has nonintersecting subcodes then it
is possible to apply switching  independently to each of the
subcodes.

From the definition of an MDS code and Proposition \ref{proNMb} we
obtain:

\begin{proposition}\label{proNM0}
Let $C\subset Q^d$ be a linear MDS code over $F$,
$(a_1,\dots,a_d)\in C$, $v\in Q\setminus \{0\}$. Then $C\cap
(L(a_1,v)\times\dots\times L(a_d,v))$ is a subcode of $C$ of order
$|F|$.
\end{proposition}

If $C_1=C\cap (L(a_1,v)\times\dots\times L(a_d,v))$ then we can
consider $C_2=C_1+(\alpha v,0,\dots,0)$, $\alpha\in F$. It is clear
that $C_2\subset L(a_1,v)\times\dots\times L(a_d,v)$ and $C_2$ has
the same parameters as $C_1$. We will say that code $(C\setminus
C_1)\cup C_2$ obtained from $C$ by a switching of type (I).

For example consider a pair of orthogonal latin squares of order $9$
below. Two subcodes (orthogonal subsquares) are marked by boldface
and italic typeface. \vskip5mm

$\begin{array}{|c|c|c|c|c|c|c|c|c|}
\hline {\bf 0}&1&2&3&{\bf 4}&5&6&7&{\bf 8} \\
\hline
1&{\it 2}&0&{4}&5&{\it 3}&{\it7}&8&{ 6} \\
 \hline
 2&0&1&5&3&4&8&6&7\\
\hline 3&{ 4}&5&{6}&7&8&0&1&{ 2} \\
\hline
{\bf 4}&5&3&7&{\bf 8}&6& 1&2&{\bf 0} \\
 \hline
5&{\it 3}&4&8&6&{\it 7}&{\it 2}&0&1 \\
\hline 6&{\it 7}&8&0&1&{\it 2}&{\it 3}&4&5 \\
\hline
7&8&6&1&2&0&4&5&3 \\
 \hline
{\bf 8}&{ 6}&7& 2&{\bf 0}&1&5&3&{\bf 4}\\
\hline
\end{array}$ \hskip10mm
$\begin{array}{|c|c|c|c|c|c|c|c|c|}
\hline {\bf 0}&1&2&3&{\bf 4}&5&6&7&{\bf 8} \\
\hline
 2&{\it 0}&1&5&3&{\it 4}&{\it 8}&6&7\\
 \hline
1&2&0&4&5&3&7&8&6 \\
\hline 6&7&8&0&1&2&3&4&5 \\
\hline
{\bf 8}&6&7& 2&{\bf 0}&1&5&3&{\bf 4}\\
 \hline
7&{\it 8}&6&1&2&{\it 0}&{\it 4}&5&3 \\
\hline 3&{\it 4}&5&6&7&{\it 8}&{\it 0}&1&2 \\
\hline
5&3&4&8&6&7&2&0&1 \\
 \hline
{\bf 4}&5&3&7&{\bf 8}&6& 1&2&{\bf 0} \\
\hline
\end{array}$
\vskip1mm

Below we can see a result of switching.

$\begin{array}{|c|c|c|c|c|c|c|c|c|}
\hline {\bf 0}&1&2&3&{\bf 4}&5&6&7&{\bf 8} \\
\hline
1&{\it 2}&0&{4}&5&{\it 7}&{\it3}&8&{ 6} \\
 \hline
 2&0&1&5&3&4&8&6&7\\
\hline 3&{ 4}&5&{6}&7&8&0&1&{ 2} \\
\hline
{\bf 4}&5&3&7&{\bf 8}&6& 1&2&{\bf 0} \\
 \hline
5&{\it 7}&4&8&6&{\it 3}&{\it 2}&0&1 \\
\hline 6&{\it 3}&8&0&1&{\it 2}&{\it 7}&4&5 \\
\hline
7&8&6&1&2&0&4&5&3 \\
 \hline
{\bf 8}&{ 6}&7& 2&{\bf 0}&1&5&3&{\bf 4}\\
\hline
\end{array}$ \hskip10mm
$\begin{array}{|c|c|c|c|c|c|c|c|c|}
\hline {\bf 0}&1&2&3&{\bf 8}&5&6&7&{\bf 4} \\
\hline
 2&{\it 0}&1&5&3&{\it 4}&{\it 8}&6&7\\
 \hline
1&2&0&4&5&3&7&8&6 \\
\hline 6&7&8&0&1&2&3&4&5 \\
\hline
{\bf 4}&6&7& 2&{\bf 0}&1&5&3&{\bf 8}\\
 \hline
7&{\it 8}&6&1&2&{\it 0}&{\it 4}&5&3 \\
\hline 3&{\it 4}&5&6&7&{\it 8}&{\it 0}&1&2 \\
\hline
5&3&4&8&6&7&2&0&1 \\
 \hline
{\bf 8}&5&3&7&{\bf 4}&6& 1&2&{\bf 0} \\
\hline
\end{array}$

\vskip5mm

Let $N(q,d,\varrho)$ be the number of  MDS codes of order $q$  with
code distance $\varrho$ and length $d$.

\begin{theorem}\label{thsqs1}
For each prime number $p$ and {\rm (a)} any  $d\leq p+1$ if
$3\leq\varrho\leq p$\\  or  {\rm (b)} arbitrary $d\geq 2$ if
$\varrho =2$ it holds
$$\ln N(p^k,d,\varrho)\geq (1+o(1))(k+m)p^{(k-2)m-1}\ln p=(1+o(1))p^{km}\ln p^k$$
as $k\rightarrow\infty$, $m=d-\varrho+1$.
\end{theorem}
\begin{proof}
Consider a linear  MDS code $C\subset Q^d$, $|Q|=p^k$, over a prime
field with rank $m$ and length $d$ (see Proposition \ref{proNMb}).
Each codeword in $C$ lies in $\frac{p^k-1}{p-1}$ different subcodes
determined by the choice of $v$ (Proposition \ref{proNM0}). The
number of different subcodes
 is equal to the product of $p^{km}$
(the cardinality of $C$) and $\frac{p^k-1}{p-1}$ divided by $p^m$
(the cardinality of subcodes). This number is greater than
$p^{k(1+m)-1}/p^m$. Each subcode intersects with
$\frac{p^k-1}{p-1}p^m$ other subcodes, which is less  than
$p^{m+k}$. Thus we can choose
$t=(1-\varepsilon(k))(p^{k(1+m)-1}/p^{2m+k})$ subcodes so that each
new subcode does not intersect with the previously selected
subcodes. There $\varepsilon(k)$ is the proportion of unimpaired
subcodes. We suppose that $\varepsilon(k)=o(1)$ and
$\ln\varepsilon(k)=o(k)$, for example, $\varepsilon(k)=\frac1k$.
Consequently, for choosing of each subcode we have more than
$w=\varepsilon(k)(p^{k(1+m)-1}/p^m)$ alternatives. By Proposition
\ref{proNMa} the code obtained by switchings of this mutually
disjoint subcodes has the same parameters as the origin code $C$.
All $t$-element sets of switchings of type (I) form different codes
because from any pair of an initial vector $(a_1,a_2,\dots,a_d)$ and
switched vector $(a_1+\alpha v,a_2,\dots,a_d)$ we can find the
switching subcode. Then $N(p^k,d,\varrho)$ is greater than
$p^tw^t/t!$, where $p$ is the number of different switchings of type
(I). Applying Stirling's formula, we get the lower bound on
$N(p^k,d,\varrho)$.
\end{proof}

\begin{proposition}\label{proNM1}{\rm \cite{KPS}}
For every integers $q,\ell,d_0$, $\ell\leq q/2$, there exists a
latin $d_0$-cube of order $q$ with a latin $d_0$-subcube of order
$\ell$.
\end{proposition}

\begin{corol}\label{thgs1}
The logarithm of the  number  of latin $d_0$-cubes of order $q$ is
$\Theta(q^{d_0}\ln q)$ as $q\rightarrow\infty$.
\end{corol}

The lower bound comes from Theorem \ref{thsqs1}  and Proposition
\ref{proNM1}, the upper bound is trivial.

\begin{proposition}\label{proNM2}{\rm \cite{Zhu}}
For every integers $q,\ell\not \in \{1,2,6\}$, $\ell\leq q/3$, there
exists a pair of orthogonal latin squares of order $q$ with
orthogonal latin subsquares of order $\ell$.
\end{proposition}

\begin{proposition}\label{proNM22}{\rm \cite{Dukes}}
For $t\geq 3$ and all sufficiently large $q,\ell$, $q\geq
8(t+1)^2\ell$, there exists a set of $t$ MOLS of order $q$ with
mutually orthogonal latin subsquares of order $\ell$.
\end{proposition}

\begin{corol}
 The logarithm of the  number of sets of $t$ MOLS of order $q$  is $\Theta(q^2\ln q)$
 as
$q\rightarrow \infty$ and $t\geq 2$ is fixed.
\end{corol}

The lower bound follows from Theorem \ref{thsqs1}  and Proposition
\ref{proNM2},\ref{proNM22}, the upper bound is trivial.

\section{Designs}

A {\it $t$-wise balanced design} $t$-BD is a pair $(X,B)$ where $X$
is a finite set of points and $B$ is a set of subsets of $X$, called
blocks, with property that every $t$-element subset of $X$ is
contained in a unique block.
 A {\it $3$-wise bipartite
balanced design} $3$-BBD($n$) (see \cite{Hartman}) is a triple
$(X,{g_1,g_2},B)$ where ${g_1,g_2}$ ($|g_1|=|g_2|$) is a partition
of $X$, $|X|=n$, $B$ is a set of $4$-element blocks such that
$|b\cap g_i|=2$ for every $b\in B, i=1,2$ with property that every
$3$-element subset $s$ ($s\cap g_i\neq \varnothing $, $i=1,2$) is
contained in a unique block.

Obviously,  Steiner system\footnote{\,We will use notation
$S(t,k,v)$, $3$-BBD($n$), etc. for sets of correspondent designs.}
$S(t,k,v)$ is a $t$-BD such that $|X|=v$ and $|b|=k$ for every $b\in
B$. Besides, Steiner quadruple system ($t=3$, $k=4$), we consider
also a $3$-BD denoted by $S(3,\{4,6\},v)$ consisting of blocks of
size $4$ or $6$.

Let $X$ be a set of points, and let $G = \{G_1,\dots,G_d\}$ be a
partition of $X$ into $d$ sets of cardinality $q$. A {\it
transverse} of $G$ is a subset of $X$ meeting each set $G_i$ in at
most
 one point. A set of $w$-element transverses of $G$ is an $H(d,
q, w, t)$ {\it design} (briefly, H-design, see \cite{HCD},VI.63) if
each $t$-element transverse of $G$ lies in exactly one transverse of
the H-design.

An MDS code $M\subset Q^d$ with code distance $t+1$ is equivalent to
$ H(d, q, d, d-t)$, where $G=\{Q_1,\dots, Q_d\}$, $Q_i$ are the
copies of $Q$, and the block $\{x_1,\dots,x_d\}$ lies in the
H-design whenever  $(x_1,\dots,x_d)\in M$.  If $t=2$ and $w=d$ an
H-design is called a {\it transversal design}. Transversal designs
are equivalent to systems of MOLS.

If $q$ is even then a $3$-BBD $(X,{g_1,g_2},B)$ is equivalent to the
MDS code $M\subset Q^4$ (with the code distance $2$)  that satisfies
the conditions
\begin{equation}\label{egs1}
 (x,y,u,v)\in M \Rightarrow (y,x,u,v), (x,y,v,u),
(y,x,v,u)\in M;\ \ \forall x,u \in Q\  (x,x,u,u)\in M.
\end{equation}
Here $g_1=Q_1\cup Q_2$, $g_2=Q_3\cup Q_4$, $Q_i$ are copies of $Q$,
and $\{x_1,x_2,x_3,x_4\}\in B$ if $(x_1,\dots,x_4)\in M$ and
$x_1\neq x_2$.


Using methods of \cite{Cameron}, \cite{KPS} and  Corollary
\ref{thgs1} we can prove the following theorem.

\begin{theorem}\label{thgs2}
The logarithm of the  number of $3$-wise bipartite balanced designs
on $n$-element set is $\Theta(n^3\ln n)$ as $n\rightarrow\infty$ and
$n$ is even.
\end{theorem}
\begin{proof} Suppose the quasigroup $f$ of order $n$ satisfies the
hypothesis of Proposition \ref{progs12} and $\ell\leq n/4$. Consider
the MDS code $M=\{ (x,y,u,v) \ |\ f(x,y)=f(u,v)\}$. It is easy to
see that $M$ meets the conditions (\ref{egs1}). Furthermore, $M$ has
subcodes $B_\sigma$ on $K_{\sigma_1}\times K_{\sigma_2}\times
K_{\sigma_3}\times K_{\sigma_4}$, where $\sigma=0101, 1001, 0110$ or
$1010$ (see Proposition \ref{progs12}).

 For any MDS code $C$  of order $\ell$
 and permutation $\pi$ we define $C_\pi=
\{(x_{\pi1},\dots,x_{\pi 4}) \ |\ x\in C\}$. Let $\Upsilon$ be a
group of permutations on 4 elements generated by transpositions
$(12)$ and $(34)$.

For any word $\sigma=(\sigma_1,\dots,\sigma_4)$ define $\pi\sigma$
as $(\sigma_{\pi1},\dots,\sigma_{\pi4})$. For each $\pi \in
\Upsilon$ we can exchange  $B_{\pi\sigma}$ by $C_\pi$ in $M$. By
Proposition \ref{proNMa} the obtained set $M'$ is an MDS code.

By construction, $M'$ satisfies (\ref{egs1}). Since we use an
arbitrary code $C$, the number of $3$-wise bipartite balanced design
is greater than the number of MDS codes of order $k$. From  Theorem
\ref{thsqs1} we obtain the lower bound of the number of designs. The
upper bound is obvious.
\end{proof}

The following doubling construction of block designs  is well known
(see \cite{Hartman}).

\begin{proposition}\label{progs3}

\noindent 1. If $S_n\in S(3,4,n)$, $B_n\in \mbox{\rm 3-BBD}(n)$ then
there exists $S_{2n}\in S(3,4,2n)$ such that $S_n,B_n\subset
S_{2n}$.

\noindent 2. If $S_n\in S(3,\{4,6\},n)$, $B_n\in \mbox{\rm
3-BBD}(n)$ then there exists $S_{2n}\in S(3,\{4,6\},2n)$ such that
$S_n,B_n\subset S_{2n}$.

\end{proposition}

\begin{proposition}\label{progs4}{\rm(\cite{Hanani63},  \cite{Hartman} Th.
4.1)} There is an injection from $S(3,\{4,6\},n)$ to
$S(3,\{4,6\},2n-2)$.
\end{proposition}

\section{Lower bound for the number of SQSs}

The following theorem provides a new construction of SQSs based on
MDS codes. Existence of  suitable MDS codes follows from
Propositions \ref{progs1} -- \ref{progs11}.

\begin{theorem}\label{thgs3}

1. If $S_{2n+2}\in S(3,4,2n+2)$, $B_n\in 3{\rm -BBD}(n)$, $n>75$ is
even, then there exists $S_{8n+2}\in S(3,4,8n+2)$ such that
$S_{2n+2},B_n\subset S_{8n+2}$.

2. If $S_{2n+2}\in S(3,\{4,6\},2n+2)$, $B_n\in 3{\rm -BBD}(n)$,
$n>75$ is even, then there exists $S_{8n+2}\in S(3,\{4,6\},8n+2)$
such that $S_{2n+2},B_n\subset S_{8n+2}$.

\end{theorem}

\begin{proof} Below we describe a construction of $S_{8n+2}$ for item 1.
Item 2 is similar.

 Let $I=\{(i,\delta) \ |\ i\in\{0,1,2,3\}, \delta\in
\{0,1\}\}$.
 Denote by $S_8$ a SQS on $I$. Let $S_{10}$ be a SQS on
 $I\cup\{e_1,e_2\}$ such that $\{(i,0),(i,1),e_1,e_2\}\in S_{10}$
 for every $i\in\{0,1,2,3\}$.
Since $n>75$, there exists an $MDS(6,8,n)$ code $M$. We enumerate
these $8$ coordinates by elements of $I$. Consider
$s=\{s_1,s_2,s_3,s_4\}\in S_8$. Denote by $M_s$ the projection of
$M$ on the coordinates $s$. By Proposition~\ref{progs1} $M_s\in
MDS(2,4,n)$. By Proposition~\ref{progs2}, there exists $C_s\in
MDS(1,4,n)$ such that $M_s\subset C_s$.

Now we will construct SQS on a set $\Omega$ where $|\Omega|=8n+2$,
$\Omega=\{e_1,e_2\}\bigcup\limits_{(i,\delta)\in I} A_{(i,\delta)}$
and $|A_{(i,\delta)}|=n$.

Consider H-designs $M^*$, $M_s^*$ and $C_s^*$ with groups
$A_{(i,\delta)}$ that correspond to MDS codes $M$, $M_s$ and $C_s$.
Let us determine quadruples  of four types.

(a) Denote $R_1=\bigcup\limits_{s\in S_8}(C^*_s\setminus M^*_s)$. It
is clear that the blocks of $\bigcup\limits_{s\in S_8}C^*_s$ cover
only once all 3-subsets of $\Omega\setminus \{e_1,e_2\}$ where three
elements lie in different groups.  Besides, a 3-subset  is covered
by a  block of $\bigcup\limits_{s\in S_8} M^*_s$ if and only if it
is included in a  8-element subset from $M^*$. Note that
$\bigcup\limits_{s\in S_8}(C^*_s)$ and $\bigcup\limits_{s\in
S_8}(M^*_s)$ are  H-designs of type $H(8,n,4,3)$ and $H(8,n,4,2)$,
respectively, on $\Omega\setminus \{e_1,e_2\}$.

(b) Consider any 8-subset $b=\{a^{i,\delta}\in A_{(i,\delta)} \ |
(i,\delta)\in I\}\in M^*$. For every $b\in M^*$ determine a set
$P_b$
 consisting of  blocks $\{a^{s_1},a^{s_2}, a^{s_3}, a^{s_4}\}$, where
$\{s_1, s_2, s_3, s_4\}\in S_{10}$ and blocks $\{a^{s_1},a^{s_2},
a^{s_3}, e_\delta\}$, where $\{s_1, s_2, s_3, \delta\}\in S_{10}$.
Denote by $R_2=\{P_b\ |\ b\in M^*\}$ the set of all these blocks. By
definition of $S_{10}$, the blocks of $R_2$ cover all 3-sets
containing $e_1$ or $e_2$ (but not both) and two elements from
$A_{(i,\delta)}$ and $A_{(i',\delta')}$ where $i\neq i'$. Moreover
the blocks of $R_1\cup R_2$ cover all 3-subsets of $\Omega\setminus
\{e_1,e_2\}$, where the three elements lie in different groups.


(c) For any pair $s_0=(i_0,\delta_0)$,  $s_1=(i_1,\delta_1)$ where
$i_0\neq i_1$ consider a $3$-BBD $B_{s_0,s_1}$ with groups $A_{s_0}$
and $A_{s_1}$. Denote $R_3=\bigcup B_{s_0,s_1}$. It is clear that a
3-subset is cover by a block of $R_3$   if and only if  two elements
of the 3-subset lie in $A_{(i,\delta)}$ and the third element lies
in $A_{(i',\delta')}$, where $i\neq i'$.

(d) For $i=0,1,2,3$ consider a Steiner quadruple systems $D_i$ on
the sets $A_{(i,0)}\cup A_{(i,1)}\cup \{e_1,e_2\}$. Define
$R_4=\bigcup D_i$.

By the construction, the blocks from $S_{8n+2}=R_1\cup R_2\cup R_3
\cup R_4$ cover any 3-subset of $\Omega$ only once. To prove
$S_{8n+2}\in S(3,4,8n+2)$, we calculate $|S_{8n+2}|$. It is well
known that SQS of order $m$ consists of $\frac{m(m-1)(m-2)}{4!}$
blocks. Therefore $|R_1|=|S_8|(n^3-n^2)=14(n^3-n^2)$, $R_2=
(|S_{10}|-4)n^2=26n^2$, $R_3=({8 \choose 2}-4)({n \choose
2}n/2)=6n^2(n-1)$, $R_4=4|S_{2n+2}|=(2n+2)(2n+1)n/3$. Then
$$|S_{8n+2}|=|R_1|+|R_2|+|R_3|+|R_4|= 20n^3+6n^2+(2n+2)(2n+1)n/3=$$
$$= 64n^3/3 + 8n^2 + 2n/3= (8n+2)(8n+1)8n/24.$$
\end{proof}

Note that it is possible to use  SQSs of order $6k+2$ and $6k+4$,
$k\geq 1$ instead of $S_8$ and $S_{10}$.

Now we obtain a lower estimate of the number of block designs as a
corollary of  Propositions \ref{progs3}(2), \ref{progs4}, Theorem
\ref{thgs3}(2) and  the asymptotic estimate from Theorem
\ref{thgs2}.

\begin{theorem}\label{thgs4}
The logarithm of the  cardinality of $S(3,\{4,6\},2n)$  is greater
than $c(n^3\ln n)$, where $c>0$ is a constant.
\end{theorem}
\begin{proof}    If $n$ is  even  then  the statement follows from
Propositions \ref{progs3}(2) and Theorem \ref{thgs2}.

If $n$  is  odd  then we will consider some cases.
 Let $2n=16k+6$. Since $16k+6=2(8k+4)-2$  the
statement follows from Proposition \ref{progs4} and the case of even
$n$. The cases $2n=16k+10=2(2(4k+4)-2)-2$ and $2n=16k+14=2(8k+8)-2$
are similar. If $2n=16k+2$ then we  use Theorems  \ref{thgs2}  and
\ref{thgs3}(2). \end{proof}

We need some known constructions of SQSs.

\begin{proposition}\label{progs5}{\rm(\cite{Hartman} Th. 4.2)}
 There is an injection from $S(3,\{4,6\},n)$ to  $S(3,4,3n-2)$.

\end{proposition}

\begin{proposition}\label{progs6}

\noindent 1. There is an injection from $S(3,4,n)$ to
$S(3,4,6n-10)$. {\rm(\cite{Hartman} Th. 4.11)}

\noindent 2. If $n\equiv 10\mod{ 12}$ then there exists an injection
from $S(3,4,n)$ to $S(3,4,3n-4)$. {\rm(\cite{Hanani} 3.4)}
\end{proposition}

The  asymptotic estimate of the number of SQSs is a corollary of
constructions of SQSs provided by Propositions \ref{progs3}(1),
\ref{progs5}, \ref{progs6}, Theorem \ref{thgs3}(1) and the
asymptotic estimates from Theorems \ref{thgs2}, \ref{thgs4}.

\begin{theorem}\label{thgs5}
The logarithm of the  cardinality of $S(3,4,n)$  is $\Theta(n^3\ln
n)$ as $n\rightarrow\infty$ and $n\equiv 2\mod{6}$ or $n\equiv
4\mod{6}$.
\end{theorem}

\begin{proof} The upper bound is obvious (see \cite{Lenz}). To prove lower
bound we will consider apart some subsequences of integers.

(a) Consider a subsequence $n=4k$. For this subsequence the required
asymptotic estimate is a corollary of Theorem \ref{thgs2} and
Proposition \ref{progs3}(1).

(b) Consider the subsequence $n\equiv4(\mod{6})$. Then $n=3(2r+2)-2$
and the required asymptotic estimate is a corollary of Theorem
\ref{thgs4} and Proposition \ref{progs5}.

It retains to  consider three subsequences $n \mod{ 36} = 2, 14\
{\rm or}\ 26$.

(c) If $n=3(12r+10)-4$ then for establishing the required asymptotic
estimate we  use Proposition \ref{progs6}(1) and the proved case
(b).

(d) If $n=6(6r+4)-10$ then  we  use Proposition \ref{progs6}(2) and
the proved case (b).

(e) Consider the case $n\mod{36} = 2$. If $n=6^4r+2= 8(3^42r)+2$
then the required  asymptotic estimate is a corollary of Theorems
\ref{thgs2} and \ref{thgs3}(1). The other cases are reduced to the
subsequence $n=6^4r+2$ by applying Proposition \ref{progs6}(2).
\end{proof}

Notice that a trivial upper bound for the cardinality of
automorphism group of $SQS(n)$ is $n!$, and $\ln(n!)=o(n^3\ln(n))$
by Stirling formula. Therefore logarithm of the number of
nonisomorphic $SQS(n)$ constructed above is asymptotically equal to
$\Theta(n^3\ln n)$.

\section{Acknowledgments}
The author would like to thank  Denis Krotov and the anonymous
referees for their useful comments.


\begin{thebibliography}{11}

\bibitem{Alek}
V. E. Alekseev, On the number of steiner triple systems,
  Mat. Zametki 15(5) (1974), 769--774 (in Russian)

\bibitem{Ball}
 S. Ball, On sets of vectors of a finite vector space in which every
subset of basis size is a basis, J. Eur. Math. Soc. 14(3) (2012),
733--748.

\bibitem{BallBeu}
S. Ball and  J. De Beule,  On sets of vectors of a finite vector
space in which every subset of basis size is a basis II, Des. Codes
Crypt. 65(1-2) (2012), 5--14.









\bibitem{Cameron}
P. J. Cameron,  Minimal edge-colourings of complete graphs,  J.
Lond. Math. Soc., Ser. II.  11(3)  (1975), 337--346.


\bibitem{Donovan}
 D. M. Donovan and  M. J. Grannell, On the number of transversal
 designs, J. Combin. Theory, Series A 120(7)  (2013), 1562--1574.

\bibitem{Dukes}
P. J. Dukes and C. M. van Bommel,  Mutually orthogonal Latin squares
with large holes, J. Statist. Plann. Inference 159 (2015), 81--89.



\bibitem{Egan}
J. Egan and I. M. Wanless, Enumeration of MOLS of small order, Math.
Comp., 85(298) (2016), 799--824.

\bibitem{Egor}
 G. P. Egorichev, Proof of the van der Waerden conjecture for
permanents, Siberian Math. J. 22 (1981), 854--859.


\bibitem{EMullen}
J. T. Ethier and  G. L. Mullen, Strong forms of orthogonality for
sets of hypercubes, Discrete Math. 312(12-13) (2012), 2050--2061.



\bibitem{Hanani}
 H. Hanani, On quadruple systems, Can. J. Math. 12  (1960), 145--157.

\bibitem{Hanani63}
 H. Hanani, On some tactical configurations, Can. J. Math. 15(4) (1963),
702--722.





\bibitem{HCD} 
Handbook of combinatorial designs. Edited by C. J. Colbourn and J.
H. Dinitz. Second edition. Discrete Mathematics and its Applications
(Boca Raton). Chapman \& Hall/CRC, Boca Raton, FL, 2007.


\bibitem{Hartman}
 A. Hartman, The fundamental constructions for 3-designs, Discrete Math.
124(1-3) (1994), 107--132 .



\bibitem{Zhu}
 K. Heinrich and L. Zhu,  Existence of orthogonal Latin squares with aligned
subsquares, Discrete Math. 59(1-2) (1986),  69-78.

\bibitem{HKO}
A. Hulpke, P. Kaski and P. R. J. Ostergard,  The number of Latin
squares of order 11, Math. Comp. 80(274) (2011), 1197--1219.

\bibitem{Keevash}
P. Keevash, Counting design, arXiv:1504.02909 (2015).

\bibitem{KKO}
J. I. Kokkala ,  D. S. Krotov and P. R. J. Ostergard,  On the
classification of MDS codes, IEEE Trans. on Inform. Theory 61(12)
(2015) 6485--6492.




\bibitem{KO}
J. I. Kokkala and P. R. J. Ostergard, Classification of Graeco-Latin
cubes,  J. Comb. Des. 23(12) (2015), 509-521.

\bibitem{KPS}
 D. S. Krotov, V. N. Potapov and P. V. Sokolova, On reconstructing reducible
$n$-ary quasigroups  and switching subquasigroups, Quasigroups and
Related Systems 16 (2008), 55--67.


\bibitem{Lenz}
H. Lenz,  On the number of Steiner quadruple systems,  Mitt. Math.
Semin. Gie\ss en 169 (1985),   55--71.



\bibitem{LR}
 C. C. Lindner and  A. Rosa, Steiner quadruple system --- a survey, Discrete
Math. 21 (1978), 147--181.



\bibitem {LL13}
 N. Linial and  Z. Luria, An upper bound on the number of Steiner triple
systems, Random Struct. Alg. 43(4) (2013), 339--406.


\bibitem {LL14}
 N. Linial and  Z. Luria, An upper bound on the number of
high-dimensional permutations, Combinatorica 34(4) (2014), 471--486.

\bibitem {LW}
 J. H.  van Lint and R. M. Wilson,  A course in combinatorics,
Cambridge U.P., 1992.

\bibitem {MS}
F. J. MacWilliams and N. J. A. Sloane, The Theory of
Error-Correcting Codes, Elsevier/North-Holland, Amsterdam, 1977.


\bibitem{MW0}
B. D. McKay and I. M. Wanless,  On the number of Latin squares, Ann.
Comb. 9(3) (2005),  335--344.

\bibitem{MW}
B. D. McKay and I. M. Wanless,  A census of small Latin hypercubes,
SIAM J. Discrete Math. 22(2) (2008), 719--736.



\bibitem{Ost}
P. R. J. Ostergard,  Switching codes and designs, Discrete Math.
312(3) (2012),  621--632.


\bibitem{KP11}
V. N. Potapov and D. S. Krotov, On the number of $n$-ary
 quasigroups of finite order,  Discrete Math. and
 Applications. 21(5-6) (2011), 575--585.


\bibitem{Ston}
D. S. Stones, The many formulae for the number of Latin rectangles,
Electron. J. Combin. 17(1) (2010),  Article 1, 46 pp.


\bibitem{Wilson}
 R. M. Wilson, Concerning the number of mutually orthogonal Latin
 squares, Discrete Math. 9(2) (1979), 181--198.


\bibitem{ZG}
 X. Zhang and G. Ge,  A new existence proof for Steiner quadruple
 systems, Des. Codes Crypt. 69(1)  (2013), 65--76.

\bibitem{ZZ}
 V. A. Zinoviev and  D. V. Zinoviev,  Non-full-rank Steiner quadruple systems
$S(v,4,3)$, Probl. Inf. Transm. 50(3) (2014), 270--279 [Prob.
Peredachi Inf. 50(3) (2014), 76--86].


\end{thebibliography}
\end{document}